\documentclass[11pt]{amsart}
\usepackage{amssymb, latexsym, graphicx, mathrsfs, enumerate, xcolor}
\usepackage{array,tabularx}

\usepackage{manfnt} 
\usepackage{amssymb,amsmath}
\makeindex

\newtheorem{thm}{Theorem}[section]

\newtheorem{lem}[thm]{Lemma}
\newtheorem{prop}[thm]{Proposition}

\theoremstyle{definition}

\theoremstyle{remark}
\newtheorem{rem}[thm]{Remark}
\newtheorem{conj}[thm]{Conjecture}
\numberwithin{equation}{section}

\newcommand*\Q{\mathbb{Q}}
\newcommand*\Z{\mathbb{Z}}

\makeatletter
\newcommand{\alignfootnote}[1]{%
	\ifmeasuring@
	\else
	\footnote{#1}%
	\fi
}
\makeatother
\begin{document}
\sloppy 
\title{Analytic ranks of elliptic curves over number fields}

\author[Peter Cho]{Peter J. Cho$^{\dagger}$}
\thanks{$^{\dagger}$ This work was supported by the National Research Foundation of Korea(NRF) grant funded by the Korea Government(MSIT)(2019R1F1A1062599) and the Basic  Science Research Program(2020R1A4A1016649)}
\address{Department of Mathematical Sciences, Ulsan National Institute of Science and Technology, Ulsan, Korea}
\email{petercho@unist.ac.kr}

\begin{abstract} Let $E$ be an elliptic curve over $\mathbb{Q}$. Then, we show that the average  analytic rank of $E$ over cyclic extensions of degree $l$ over $\Q$ with $l$  a prime not equal to $2$, is at most $2+r_\Q(E)$, where $r_\Q(E)$ is the analytic rank of the elliptic curve $E$ over $\Q$. This bound is independent of the degree $l$  Also, we also obtain some average analytic rank results over $S_d$-fields.
\end{abstract}

\maketitle

\section{Introduction}
Let $E$ be an elliptic curve defined over $\mathbb{Q}$ with conductor $Q_E$. For a number field $F$, let $E(F)$ be the group of  $F$-rational points of the elliptic curve $E$. Let $L_F(s,E)$ be the normalized $L$-function of $E$ over the field $F$ so that its central point is $\frac 12$. We omit the subscript $F$  from $L_F(s,E)$ when $F$ is the field of rational numbers. We are interested in the behavior of the analytic ranks of the $L$-functions $L_F(s,E)$ when $F$ is a cyclic extension of prime degree $l$ over $\Q$. For a prime $l\geq 2$, we denote the family of all cyclic extensions $F$ of degree $l$ over $\Q$ by $C_l$. Then, for  a number field $F$ in $C_l$ we have
\begin{align}\label{decomp}
L_F(s,E)=L(s,E) \prod_{\chi} L(s, E \times \chi),
\end{align}
where $\chi$ runs over the $(l-1)$ primitive $l$-th order Dirichlet characters corresponding to the field $F$. 

For quadratic fields $F$, the average analytic rank of $L(s,E\times \chi)$ is expected to be $\frac{1}{2}$ regardless of the analytic rank $r_\Q(E)$ of $L(s,E)$ by Goldfeld's conjecture. So we have the following statement equivalent to Goldfeld's conjecture \cite{Go}.

\begin{conj}[Goldfeld's conjecture] Let $E$ be an elliptic curve over $\Q$. Then, the average analytic rank of $E$ over quadratic fields $F$ is $\frac{1}{2}+r_\Q(E)$. 
\end{conj}

Goldfeld's conjecture says that a half of the twisted $L$-functions $L(s, E\times \chi)$ do not vanish at the central point and the other half of them vanish to order $1$ at the central point. However, the story seems different for cyclic extensions of prime degree $l\geq 3$. David, Fearnley and Kisilevsky \cite{DFK} conjectured that for a fixed elliptic curve $E$ and a fixed prime $l\geq 7$, there are only finitely many primitive Dirichlet characters $\chi$ of order $l$ for which $L(s,E \times \chi)$ vanishes at the central point $s=1/2$. Also, they conjectured that only a small number of twisted $L$-functions $L(s,E\times \chi)$ vanish for $l=3$ and $5$. As a direct consequence of the conjecture, we have the following conjecture.

\begin{conj} \label{conj2} Let $l\geq 3$ be a prime and $E$ an elliptic curve over $\Q$. Then, the average analytic rank of $E(F)$ over the family $C_l$ is $r_\Q(E)$. 
\end{conj}

We can relate this conjecture with Diophantine Stability introduced by Mazur and Rubin \cite{MR} recently. Let $K$ be a number field. Suppose $V$ is an irreducible algebraic variety over $K$. If $L$ is a field containing $K$, we say that $V$ is diophantine-stable for $L/K$ if $V(L)=V(K)$. For a given elliptic curve  $E$ over $\Q$, if $r_F(E)=r_\Q(E)$ for a number field $F$, then under the Birch and Swinnerton-Dyer conjecture, the algebraic rank $r_F^{\text{alg}}(E)$ of $E(F)$ is equal to the algebraic rank $r_\Q^{\text{alg}}(E)$ of $E(\Q)$. By Merel's uniform bound \cite[Theorem 7.5.1]{Sil} on the size of $E(F)_{\text{tor}}$, we can see  that there are only finitely many number fields $F$ of degree $l$ for which $E(F)_{\text{tor}} > E(\Q)_{\text{tor}}$. Even if $r_F^{\text{alg}}(E)=r_\mathbb{Q}^{\text{alg}}(E)$ and $E(F)_{\text{tor}} = E(\Q)_{\text{tor}}$, there could be a $F$-rational point not belonging to $E(\mathbb{Q})$ unfortunately. \footnote{For example, consider an elliptic curve $E$ which is given by the Weierstrass equation $y^2=x^3+9$. Then $E(\Q)\cong \Z \times \Z/3\Z$. Let $K$ be the field by adjoining a root $\alpha$ of $x^4+8x^3-72x+72$ to $\Q$. Then $E(K)\cong \Z \times \Z/3\Z$ and $E(K)$ contains a new point $P=(\alpha, -\frac{1}{2}\alpha^3-3\alpha^2+9)$ of infinite order such that $2P=(-2,1)$. } However, these results still give strong conjectural evidence that an elliptic curve $E$ over $\mathbb{Q}$ is diophantine-stable for $L/\mathbb{Q}$ for most cyclic fields $L$ of prime degree $l \geq 3$.

%
%

We can understand these two seemingly different phenomena through Katz and Sarnak's $n$-level density conjecture for families of $L$-functions. Their philosophy is that the distribution of low-lying zeros of $L$-functions in a natural family is governed by one of the five classical matrix groups $O, SO(even), SO(odd), USp,$ and $U$, which we call the symmetry type of the family. We refer to \cite{Ru} for the introduction of the conjecture.  

From a work of Rubinstein \cite{Ru}, when $\chi$ is quadratic, we can see that the symmetry type for the family of $L$-functions $L(s, E \times \chi)$ is $O$ and the average of analytic ranks $r_F(E)$ is at most $2.5=2+0.5$. Heath-Brown \cite{HB} lowered the bound to $1.5$. If Katz and Sarnak's one-level conjecture is true for a test function with arbitrarily large compact support, the average analytic rank would be $1/2$, which is Goldfeld's conjecture. 


In \cite{CP2019}, the author and Park computed the one-level density for families of $L$-functions $L(s, \pi \times \chi)$  for a cuspidal representation $\pi$ of $GL_M(A_\mathbb{Q})$. From it, when $\chi$ is a primitive character of prime degree $l>2$, we can determine that the symmetry type for the family of $L(s, E \times \chi)$ is $U$. Under the one-level density conjecture for the symmetry type $U$, the average analytic rank becomes $r_\Q(E)$.  

In this article, we make partial progress toward Conjecture \ref{conj2}. Let $E$ be an elliptic curve defined over $\mathbb{Q}$, and $F$  a field in $C_l$. By $(\ref{decomp})$, we have
\begin{align*}
 r_F(E) = r_\mathbb{Q}(E) + \sum_{\chi} \text{ord}_{s=1/2}L(s, E \times \chi).
\end{align*}
Hence the average of $r_F(E)$ is given by
\begin{align*}
r_\mathbb{Q}(E) +  \mbox{the average of $\text{ord}_{s=1/2}L(s, E \times \chi)$},
\end{align*}
where the average is taken over some subfamily of $C_l$ which we describe now. 
We consider primitive characters with conductor $q_\chi$ coprime to $Q_E$. The condition $(q_\chi,Q_E)=1$ determines the (not analytic but ordinary) conductor $q(E \times \chi)$ of $L(s, E \times \chi)$ completely, which is $q_\chi^2Q_E$ by a work of Barthel and Ramakrishnan \cite{BR}. 

Define
\begin{align*}
{C}_{l, Q_E}= \{F  \in C_l | (q_F, Q_E)=1 \}
\end{align*}
where $q_F$ is the conductor of the field $F$.
Now let $\omega$ be a non-negative Schwartz class function. We define
\begin{align*}
 \frak{C}_{l,Q_E}(X)=\sum_{F \in C_{l, Q_E}} \omega \left( \frac{q_F}{X}\right).
\end{align*}
 
We show that the average analytic rank has a nice uniform upper bound independent of the degree $l$.
\begin{thm}\label{main} Assume GRH.\footnote{We need GRH for the following $L$-functions: $\zeta(s)$, Dirichlet $L$-functions with $\chi$ mod $2l$, Hecke $L$-functions over $K=\Q(\zeta_l)$ with characters of order $l$, $ \zeta_K(s)$ and $L(s, E\times \chi)$ for primitive Dirichlet characters of order $l$. } 
Let $E$ be an elliptic curve over $\mathbb{Q}$, $l$ be a prime $\geq 3$. Then, 
\begin{align*}
\lim_{X \rightarrow \infty}\frac{\sum_{F \in C_{l,Q_E}(X)}r_F(E)\omega \left(\frac{q_F}{X} \right)}{ \frak{C}_{l,Q_E}(X) } \leq 2 + r_\mathbb{Q}(E).
\end{align*}
\end{thm}

In Section \ref{S_d}, we also give an upper bound on the average analytic rank over some non-abelian fields. A number field $F$ of degree $d$ is an $S_d$-field if its normal closure $\widehat{F}$ over $\Q$ is an $S_d$ Galois extension.  For example, quadratic fields are $S_2$-fields. 

For an $S_d$-field $F$ we have 
\begin{align*}
L_F(s,E)=L(s,E)L(s,E \times \rho),
\end{align*}
where $\rho$ is the $(d-1)$-dimensional standard representation of the symmetry group $S_d$. 

Let $\mathcal{S}_{d,Q_E}$ be the family of $S_d$-fields $F$ with discriminant $D_F$ coprime to $Q_E$ and $\mathcal{S}_d$ be the family of $S_d$-fields with no restriction on discriminant. For a positive number $X$, let
$$
\mathcal{S}_{d,Q_E}(X)=\{F \in \mathcal{S}_{d,Q_E}| |D_F|\leq X \},
$$
where $D_F$ is the discriminant of the field $F$.  In \cite{OT}, Lemke Oliver and Thorne showed that there is a constant $c_d>0$ such that in $\mathcal{S}_d(X)$, for any $\epsilon>0$, there are $\gg_{E,\epsilon} X^{c_d-\epsilon}$ $S_d$-fields $F$ with $r_F(E) > r_\Q(E)$.  

For $S_3$-fields, using a recent result of Bhargava, Taniguchi and Thorne \cite{BTT} we have our second main result.
\begin{thm} \label{thm2} Assume GRH. Let $E$ be an elliptic curve over $\Q$. 
The average analytic rank $r_F(E)$ over $S_{3,Q_E}$ is bounded by $7.5+r_\Q(E)$.
\end{thm}

\begin{rem}
We also have an analogue of Theorem \ref{thm2} for $S_4$-fields and $S_5$-fields, which are mentioned at the end of Section \ref{S_d}. These poor bounds are due to the poor error term of the counting functions $(\ref{count})$ for $S_4$-fields and $S_5$-fields. See \cite{CKcnt}.
\end{rem}

In Section \ref{EF-sec}, we introduce an explicit formula we use which is one of the main tools for one-level density. In Section \ref{cyclic fields}, we recall some preliminaries on primitive Dirichlet characters and lemmas for proof of Theorem \ref{main}. Sections \ref{cyclic-l} and \ref{S_d} are devoted to the proof of Theorems \ref{main} and \ref{thm2}.  

\section{Explicit formula} \label{EF-sec}

Let $L(s,f)$ be an entire $L$-function with conductor $q(f)$ and gamma factor $\gamma(f,s)$ which satisfies the standard functional equation:
\begin{align*}
\Lambda(s, f)=q(f)^\frac{s}{2}\gamma(f,s)L(s,f)=\omega_f\Lambda(1-s, \overline{f}),
\end{align*}
where $\omega_f$ is the root number of modulus $1$. Let $\Lambda_f(n)=a_f(n)\Lambda(n)$ be the $n$-th coefficient of the Dirichlet series $-\frac{L'}{L}(s,f)=\sum_{n=1}^\infty \frac{\Lambda_f(n)}{n^s}$. If the Euler factor of $L(s,f)$ at the place $p$ is $\prod_{i=1}^d\left(1-\frac{\alpha_i(p)}{p^s} \right)^{-1}$, then $a_f(p^k)=\sum_{i=1}^d \alpha_i(p)^k$ and $\Lambda_f(p^k)=a_f(p^k)\log p$. By \cite[Theorem 5.12]{IK}, we have the following explicit formula.   
\begin{lem} Let $\phi$ be an even Schwartz class function such that its Fourier transform $\widehat{\phi}$ is compactly supported. Let $L(s,f)$ be an $L$-function as above. For a  parameter $L>0$, we have  
\begin{align*}
\sum_{\rho=\frac{1}{2}+i\gamma} \phi \left( \gamma \frac{\log L}{2\pi}\right) = \widehat{\phi}(0)\frac{\log q(f)}{\log L}- \frac{1}{\log L}\sum_{n}\left(\frac{\Lambda_f(n)}{\sqrt{n}}+\frac{\Lambda_{\overline{f}}(n)}{\sqrt{n}} \right) \widehat{\phi}\left(\frac{\log n}{\log L} \right)\\
+ \frac{1}{2\pi }\int_{-\infty}^\infty \left( \frac{\gamma'}{\gamma}(f,\frac 12 + it) +\frac{\gamma'}{\gamma}(\overline{f},\frac 12 - it) \right) \phi \left( \frac{t \log L}{2\pi}\right)dt,
\end{align*}  
where the sum is over non-trivial zeros $\rho=\frac{1}{2}+i\gamma$ of $L(s, f)$ with multiplicity. 
\end{lem}

We can show using \cite[Lemma 12.14]{MV} that
\begin{align*}
 \frac{1}{2\pi }\int_{-\infty}^\infty \left( \frac{\gamma'}{\gamma}(f,\frac 12 + it) +\frac{\gamma'}{\gamma}(f,\frac 12 - it) \right) \phi \left( \frac{t \log L}{2\pi}\right)dt \ll \frac{1}{\log L}.
\end{align*}
We assume that $f$ satisfies the Ramanujan-Petersson conjecture. This assumption is true for the $L$-functions we consider. Then we have   
\begin{align*}
\frac{1}{\log L}\sum_{n=p^k, k\geq 3}\left(\frac{\Lambda_f(n)}{\sqrt{n}}+\frac{\Lambda_{\overline{f}}(n)}{\sqrt{n}} \right) \widehat{\phi}\left(\frac{\log n}{\log L} \right) \ll \frac{1}{\log L},
\end{align*}
by absolute convergence of the Dirichlet series. Therefore, 
\begin{align} \label{EF}
\sum_{\gamma} \phi \left( \gamma \frac{\log L}{2\pi}\right) = \widehat{\phi}(0)\frac{\log q(f)}{\log L}- \frac{1}{\log L}\sum_{p}\left(\frac{\Lambda_f(p)}{\sqrt{p}}+\frac{\Lambda_{\overline{f}}(p)}{\sqrt{p}} \right) \widehat{\phi}\left(\frac{\log p}{\log L} \right)\\
- \frac{1}{\log L}\sum_{p}\left(\frac{\Lambda_f(p^2)}{p}+\frac{\Lambda_{\overline{f}}(p^2)}{2} \right) \widehat{\phi}\left(\frac{2\log p}{\log L} \right) + O\left( \frac{1}{\log L} \right). \nonumber
\end{align}  
We will use $(\ref{EF})$ for our one-level density computation. 

\section{Cyclic extensions of degree $l$} \label{cyclic fields}
For a prime $l\geq 3$, let $F$ be a cyclic extension of degree $l$ over $\Q$. There is an $(l-1)$-to-1 correspondence between primitive Dirichlet characters $\chi$ of order $l$ and cyclic extensioms $F$ of degree $l$ over $\Q$. Thus, counting cyclic extensions of degree $l$ over $\Q$ can be reduced to counting primitive Dirichlet characters. 

In \cite{CP2019}, the author and Park summarize the following well-known results for primitive Dirichlet characters of prime order $l$. 
\begin{prop}\label{character2}
Assume that $l$ is a prime. 
\begin{enumerate}
\item When $l=2$, $q_\chi$ is the conductor of a primitive quadratic character $\chi$ if and only if $q_\chi=2^b m$ where $m$ is an odd square-free integer and $b=0$, $2$ or $3$.  
\item When $l>2$, $q_\chi$ is the conductor of a primitive character of order $l$ if and only if 
$$
q_\chi=l^b \prod_{ q \equiv 1 \mod l }^{\text{finite}}q, \quad \mbox{$b$=$0$ or $2$}.
$$
\item Let $q$ be the conductor of a primitive character of order $l$ with $\gcd(q, l)=1$. Then, the number of primitive characters of order $l$ with conductor $q$ is $(l-1)^{\omega(q)}$, where $\omega(n)$ is the number of  distinct prime divisors of $n$.
\end{enumerate}
\end{prop}

\begin{rem}
Since there are $(l-1)$ primitive Dirichlet character of order $l$ with conductor $l^2$, we can see that  the number of primitive characters of order $l$ with conductor $q$ is also $(l-1)^{\omega(q)}$.
\end{rem}

Recall that $E$ is an elliptic curve over $\mathbb{Q}$ with conductor $Q_E$.  We want to count the fields in $C_{\,Q_E}$ by considering the primitive Dirichlet characters of order $l$ with conductor $q_\chi$ coprime to $Q_E$. This can be achived by the following generating series:
\begin{align*}
\left( 1+ \frac{(l-1)}{l^{2s}} \right)^{1-\delta_{l|Q_E}} \prod_{p\equiv 1 \mbox{ (mod $l$)},p \nmid Q_E} \left( 1+ \frac{(l-1)}{p^s}\right)=\sum_{q=1}^\infty \frac{a(q)}{q^s}, 
\end{align*} 
where $a(q)$ is the number of primitive Dirichlet characters $\chi$ of order $l$ with conductor $q_\chi$ coprime to $Q_E$, which is $(l-1)^{\omega(q) }$. In \cite{CP2019}, we showed that the Dirichlet series $\prod_{p\equiv 1 \mbox{ (mod $l$)}} \left( 1+ \frac{(l-1)}{p^s}\right)$ has meromorphic continuation for $\Re(s)>1/4$ with a simple pole at $s=1$ and a pole of a finite order at $s=1/3$. Since the term $\frac{p^s}{p^s+(l-1)}$ for a prime divisor $p$ of $Q_E$ congruent to $1$ modulo $l$ has poles on the line $Re(s)=\log_p(l-1)<1$, the Dirichlet series $\sum_{q=1}^\infty a(q)q^{-s}$ is meromorphic with a simple pole at $s=1$ for $\Re(s)>H_{Q_E}$ for some constant $H_{Q_E}$ with $\frac{1}{3}<H_{Q_E} <1$. 

We count the primitive characters with a weight. Let $\omega$ be a  non-negative Schwartz class function. Then, we define
\begin{align*}
W_{Q_E}(X)=\sum_\chi^* \omega\left( \frac{q_\chi}{X}\right)=\sum_{q}\omega\left( \frac{q}{X}\right)a(q),
\end{align*}
where the first sum is over all primitive characters of order $l$ with conductors $q_\chi$ coprime to $Q_E$. 

\begin{lem} \label{lem1} Under GRH, for any $\varepsilon>0$
\begin{align*}
W_{Q_E}(X)=R_{\omega,l, Q_E}X+O_{\omega,l,Q_E,\varepsilon}(X^{H_{Q_E}+\varepsilon}) \mbox{ for some constant $R_{\omega,l,Q_E}$}.
\end{align*}
\end{lem}
\begin{proof}
The proof is essentially the same as that of \cite[Lemma 3.8]{CP2019}.
\end{proof}

\begin{lem} \label{con-sum} Under GRH,
\begin{align*}
\sum_\chi^* \omega\left( \frac{q_\chi}{X}\right) \log q_\chi=  W_{Q_E}(X) \log X + O_{\omega,l,Q_E}(X).
\end{align*}
\end{lem}
\begin{proof}
The proof is essentially the same as that of \cite[Lemma 3.5]{CP2019}.
\end{proof}

\begin{lem} \label{lem2} Under GRH when $n$ is not a $l$-th power,
\begin{align*}
\sum_\chi^* \omega\left( \frac{q_\chi}{X}\right)\chi(n)\ll_{\omega, \epsilon} n^\epsilon X^{1/2+\epsilon}.
\end{align*}
\end{lem}
\begin{proof}
The proof is essentially the same with that of \cite[Lemma 3.9]{CP2019}.
\end{proof}

\section{Proof of Theorem \ref{main}} \label{cyclic-l}

Let $f$ be the modular form of weight $2$ with level $Q_E$ which corresponds to the elliptic curve $E$ and $\chi$ a primitive Dirichlet character of order $l$ with conductor $q_\chi$ coprime to $Q_E$.  Then, the conductor $q(f \times \chi)$ of $L(s, f \times \chi)(=L(s, E\times \chi))$ is exactly $Qq_\chi^2$ by a work of Barthel and Ramakrishnan \cite{BR}. 

The one-level density for an $L$-function $L(s, f \times \chi)$ is defined to be 
\begin{align*}
D_X(f \times \chi, \phi) &= \sum_{\gamma_{f \times \chi}} \phi \left( \gamma_{f\times \chi} \frac{\log L}{2\pi} \right),
\end{align*}
where $\gamma_{f \times \chi}$ denote the imaginary part of a generic non-trivial zero and $L=X^2$ for a parameter $X$. \footnote{In place of $L$ in the one-level density, there should be a parameter which is of the same order of magnitude as the analytic conductor of $L(s, f\times \chi)$. In our case, the conductor $q(f\times \chi)$ of $L(s, f \times \chi)$ is of the same order of magnitude as the corresponding analytic conductor and $L=X^2$ is of the same order of magnitude as the conductor $q(f \times \chi)$.} 
Let $\phi$ be an even 
Schwartz class function such that its Fourier transform $\widehat{\phi}$ is compactly supported. Then, by Weil's explicit formula $(\ref{EF})$, we have
\begin{align*}
D_X(f \times \chi, \phi) &= \widehat{\phi}(0)\frac{\log c_{f \times \chi} }{\log L} -\frac{1}{\log L}\sum_{p} \frac{\log p}{p^{1/2}}\left( a_{f \times \chi}(p)\widehat{\phi}\left( \frac{\log p}{\log L} \right) + a_{f \times \overline{\chi}}(p)\widehat{\phi}\left( \frac{\log p}{\log L} \right)\right)\\
&  -\frac{1}{\log L}\sum_{p} \frac{\log p}{p}\left( a_{f \times \chi}(p^2)\widehat{\phi}\left( \frac{2\log p}{\log L} \right) + a_{f \times \overline{\chi}}(p^2)\widehat{\phi}\left( \frac{2\log p}{\log L} \right)\right) + O\left( \frac{1}{\log L}\right)\\
&= \widehat{\phi}(0)\frac{2 \log q_\chi }{\log L} -\frac{1}{\log L}\sum_{p} \frac{\log p}{p^{1/2}}\left( a_{f \times \chi}(p)\widehat{\phi}\left( \frac{\log p}{\log L} \right) + a_{f \times \overline{\chi}}(p)\widehat{\phi}\left( \frac{\log p}{\log L} \right)\right)\\
&  -\frac{1}{\log L}\sum_{p} \frac{\log p}{p}\left( a_{f \times \chi}(p^2)\widehat{\phi}\left( \frac{2\log p}{\log L} \right) + a_{f \times \overline{\chi}}(p^2)\widehat{\phi}\left( \frac{2\log p}{\log L} \right)\right) + O\left( \frac{1}{\log L}\right).
\end{align*}
Note that $a_{f \times \chi }(n) =a_f(n) \times \chi(n)$. 

Hence, we have
\begin{align*}
\frac{1}{W_{Q_E}(X)}\sum_{\chi}^*D_X(f \times \chi, \phi)\omega\left( \frac{q_\chi}{X}\right) & = \frac{\widehat{\phi}(0)}{W_{Q_E}(X)} \sum_{\chi}^* \frac{2 \omega(q_\chi/X) \log q_\chi}{\log L} +S_1+S_2+ O\left( \frac{1}{\log L}\right), 
\end{align*}
where
\begin{align*}
&S_1 =-\frac{2}{W_{Q_E}(X)\log L}\sum_{p}\frac{\log p}{p^{1/2}}a_f(p)\widehat{\phi}\left( \frac{\log p}{L} \right)\left(\Re\sum_\chi^* \chi(p)\omega\left( \frac{q_\chi}{X}\right)\right),\\
&S_2 =-\frac{2}{W_{Q_E}(X)\log L}\sum_{p}\frac{\log p}{p}a_f(p^2)\widehat{\phi}\left( \frac{2\log p}{L} \right)\left(\Re\sum_\chi^* \chi(p^2)\omega\left( \frac{q_\chi}{X}\right)\right).
\end{align*}

By Lemma \ref{con-sum}, the first sum is $\widehat{\phi}(0)+ O\left( \frac{1}{\log L} \right)$. 
Now we assume that  $\widehat{\phi}$ is supported in $(-1/2,1/2)$. Then, by Lemma \ref{lem2},
\begin{align*}
&S_1 \ll \frac{1}{X \log X} \sum_{ p < X^{1-2\epsilon}}\frac{p^{\epsilon}\log p}{p^{1/2}}X^{1/2+\epsilon} \ll \frac{1}{\log X},\\
&S_2 \ll \frac{1}{X \log X}  \sum_{ p < X^{1/2-\epsilon}}\frac{p^{\epsilon}\log p}{p}X^{1/2+\epsilon} \ll \frac{1}{\log X}.
\end{align*}

\begin{thm} \label{one-level} Let $\phi$ be an even Schwartz class function such that  its Fourier  $\widehat{\phi}$ is supported in $(-1/2,1/2)$. Then,
\begin{align*}
\lim_{X \rightarrow \infty} \frac{1}{W_{Q_E}(X)}\sum_{\chi}^*  \omega \left( \frac{q_\chi}{X}\right) D_X(f \times \chi, \phi)  =  \widehat{\phi}(0).
\end{align*}
\end{thm}

Let  $r_{E,\chi}$ denote the analytic rank of $L(s, f\times \chi)$. If $\phi$ is a non-negative valued function with $\phi(0)>0$,  by a trivial bound
\begin{align*}
r_{E,\chi}\phi(0) \leq  \sum_{\gamma_{f \times \chi}} \phi \left( \gamma_{f\times \chi} \frac{L}{2\pi} \right), 
\end{align*}
we have
\begin{align*}
 \frac{\phi(0)}{W_{Q_E}(X)}\sum_{\chi}^*  r_{E,\chi}\omega \left( \frac{q_\chi}{X} \right) \leq  \frac{1}{W_{Q_E}(X)}\sum_{\chi}^* D_X(f \times \chi, \phi) \omega\left( \frac{q_\chi}{X}\right),
\end{align*}
and it implies
\begin{align*}
\lim_{X \rightarrow \infty} \frac{1}{W_{Q_E}(X)}\sum_{\chi}^*  r_{E,\chi}\omega \left( \frac{q_\chi}{X} \right)
=\lim_{X \rightarrow \infty} \frac{\sum_{F \in C_{l,Q_E}(X)}(r_F(E)-r_\Q(E))\omega \left( \frac{q_F}{X} \right)}{ \frak{C}_{l,Q_E}(X)} \leq \frac{\widehat{\phi}(0)}{\phi(0)}.
\end{align*}

In particular, we take $\phi(x)=\frac{\sin^2 \left( 2\pi \frac{1}{2} \sigma x \right)}{(2\pi x)^2}$. Then, 
\begin{align*}
\widehat{\phi}(u)=\frac{1}{2}\left( \frac{1}{2} \sigma - \frac{1}{2} |u| \right) \mbox{ for $|u|\leq \sigma$, $\phi(0)=\frac{\sigma^2}{4}$
, and $\widehat{\phi}(0)=\frac{\sigma}{4}$.}
\end{align*}
By choosing $\sigma=1/2$, Theorem \ref{main} follows. 

\section{$S_d$-fields} \label{S_d}
%

The main tool for the one-level density of the family of elliptic curve $L$-functions over $S_d$-fields is counting number fields with a finite number of local conditions. First, we introduce some notation and known results. 
Let $C$ denote a conjugacy class of the group $S_d$, and $r_1,r_2,\cdots,r_w$ be the possible splitting types of a prime in a $S_d$-field. We say that an $S_d$-field $F$ satisfies the local condition $\it S_{p,C}$ if $p$ is unramified in $F$ and the conjugacy class of Frobenius automorphism at $p$ is $C$. An $S_d$-field $F$ is said to satisfy the local condition $\it S_{p,r_i}$ if $p$ is ramified in $F$ and its splitting type is $r_i$. 

Let $\mathcal S=(\it{LC}_{p_i})_{i=1}^k$ be a finite set of local conditions. Define the density  of the set $\mathcal S$ by
\begin{align*}
|\it S_{p,C}|=\frac{|C|}{|S_d|(1+f(p))}, \quad |\it S_{p,r_i}|=\frac{c_i(p)}{1+f(p)}, \quad |S|=\prod_{i=1}^k |\it{LC}_{p_i}|
\end{align*}
for some positive-valued functions $f(p)$ and $c_i(p)$ on the set of primes with $\sum_{i}c_i(p)=f(p)$. Note that the functions $f(p)$ and $c_i(p)$ depend on the group $S_d$. For $S_3$-fields, there are two splitting types for a ramified prime in a $S_d$-field, which are partial ramification and total ramification and we denote them by $r_{1}$ and $r_{2}$ respectively. Then $c_{1}(p)=\frac{1}{p}$, $c_{2}(p)=\frac{1}{p^2}$, and $f(p)=\frac{1}{p}+\frac{1}{p^2}$. 

Let
\begin{align*}
S_d(X,\mathcal S)=\{ F \in S_d | |D_F| \leq X, \text{$F$ satisfies $\mathcal S$} \}. 
\end{align*} 
For $d=3,4,$ and $5$, the cardinality of $S_d(X,\mathcal S)$ can be estimated with a power saving error term \cite{BBP, TT, ST, Y, CKcnt}:
\begin{align} \label{count}
\left| S_d(X,\it S) \right| = c_{d}|\mathcal S|X + O\left((\prod_{i=1}^k p_i)^{\alpha_d}) X^{1-\delta_d}\right),
\end{align}
for some positive constants $c_d$, $\alpha_d$ and $0<\delta_d<1$ which depends on $S_d$.

Then, we can compute the cardinality of $S_{d,Q_E}(X)$ by forcing all the prime divisors $p$ of $Q_E$ not to ramify.  From now on, we focus on the case $d=3$. For $S_3$-fields, due to a recent work of Bhargava, Taniguchi and Thorne \cite{BTT}, we have
\begin{align} \label{s3-count}
\left| S_3(X,\mathcal S) \right| = c_{3}|S|X + O\left(X^\frac{5}{6}+ (\prod_{i=1}^k p_i)^\frac{2}{3} X^{\frac{2}{3}+\epsilon}\right).
\end{align}
By $(\ref{s3-count})$, we have
\begin{align}\label{F-size}
|S_{3,Q_E}| = c_{3,Q_E}X + O_{Q_E}\left(X^\frac{5}{6}+ X^{\frac 23 + \epsilon} \right)
\end{align}
where $c_{3,Q_E}=\left( \prod_{q | Q_E}\frac{1}{1+f(p)} \right) c_3$. 

We define the one-level density for $L(s, f \times \rho)$ by  
\begin{align*}
D_X(f \times \rho, \phi) = \sum_{\gamma_{f \times \rho}} \phi \left( \gamma_{f\times \rho} \frac{\log L}{2\pi} \right)
\end{align*}
where $\gamma_{f \times \chi}$ denote the imaginary part of a generic non-trivial zero and $L=X^2$. 

Once we show that for $supp(\widehat{\phi})\subset [-\sigma, \sigma]$ with $\sigma< \frac{1}{7}$ 
\begin{align}\label{one-level-S-d}
\frac{1}{|S_{d,Q_E}(X)|}\sum_{F \in S_{d,Q_E}(X)}D_X(f \times \rho, \phi)=\widehat{\phi}(0)+\frac{\phi(0)}{2}+ O\left( \frac{1}{\log X} \right),
\end{align}
we have 
\begin{align*}
\frac{\phi(0)}{|S_{d,Q_E}(X)|}\sum_{F \in S_{d,Q_E}(X)}r_{ E \times \rho}=\frac{\phi(0)}{|S_{d,Q_E}(X)|}\sum_{F \in S_{d,Q_E}(X)}{(r_F(E)- r_\Q(E))} \leq  \widehat{\phi}(0)+\frac{\phi(0)}{2}+ O\left( \frac{1}{\log X} \right).
\end{align*}

By taking 
\begin{align*}
\phi(x)=\frac{\sin^2 \left( 2\pi \frac{1}{2} \sigma x \right)}{(2\pi x)^2}, \,\widehat{\phi}(u)=\frac{1}{2}\left( \frac{1}{2} \sigma - \frac{1}{2} |u| \right) \mbox{ for $|u|\leq \sigma$}
\end{align*}
with $\sigma=\frac{1}{7}$, Theorem \ref{thm2} follows. 

\begin{rem}
The one-level density in $(\ref{one-level-S-d})$ is different from that of Theorem \ref{one-level}, which means that the symmetry types for the two families are different. The symmetry type of the former one is $U$ and the symmetry type of the latter one is $O$.
\end{rem}

Now, it is left to show $(\ref{one-level-S-d})$. 
Since the conductor $q(f \times \rho)$ of $L(s, f \times \rho)$ is $|D_F|^2Q_E$ \cite{BR}, by the Explicit formula $(\ref{EF})$, we have
 \begin{align*}
D_X(f \times \rho, \phi) 
&= \widehat{\phi}(0)\frac{2 \log |D_F| }{\log L} -\frac{2}{\log L}\sum_{p} \frac{\log p}{p^{1/2}} a_{f \times \rho}(p)\widehat{\phi}\left( \frac{\log p}{\log L} \right) \\
&  -\frac{2}{\log L}\sum_{p} \frac{\log p}{p} a_{f \times \rho}(p^2)\widehat{\phi}\left( \frac{2\log p}{\log L} \right)+ O\left( \frac{1}{\log L}\right).
\end{align*}

Since $$a_{f \times \rho}(p)=a_{f}(p)a_\rho(p), \quad a_{f \times \rho}(p^2)=a_{f}(p^2)a_\rho(p^2),$$
we have
\begin{align} \label{S_d one-level}
\frac{1}{|S_{d,Q_E}(X)|}\sum_{F \in S_{d,Q_E}(X)}D_X(f \times \rho, \phi)&=\frac{2 \widehat{\phi}(0)}{|S_{d,Q_E}(X)|\log L}\sum_{F \in S_{d,Q_E}(X)} \log |D_F|\\
&+S_1+S_2+O\left( \frac{1}{\log L} \right),
\end{align}
where 
\begin{align*}
&S_1=-\frac{2}{|S_{d,Q_E}(X)|\log L}\sum_{p} \frac{\log p}{p^{1/2}} a_{f}(p)\widehat{\phi}\left( \frac{\log p}{\log L} \right)\left( \sum_{F \in S_{d,Q_E}(X)} a_\rho(p) \right),\\
&S_2=-\frac{2}{|S_{d,Q_E}(X)|\log L}\sum_{p} \frac{\log p}{p} a_{f}(p^2)\widehat{\phi}\left( \frac{2 \log p}{\log L} \right)\left( \sum_{F \in S_{d,Q_E}(X)} a_\rho(p^2) \right).
\end{align*}

We can determine $a_\rho(p)$ and $a_\rho(p^2)$ by the corresponding conjugacy class $C$ and it is summarized in the table below. 
\begin{center}
\begin{tabularx}{0.8\textwidth} { 
  | >{\centering\arraybackslash}X 
  | >{\centering\arraybackslash}X 
  | >{\centering\arraybackslash}X | }
 \hline
 Conjugacy class & $a_\rho(p)$ & $a_\rho(p^2)$ \\ 
\hline
 (1)& 2 & 2 \\ 
 (12) & 0 & 2 \\ 
(123) &  -1 & -1 \\
 \hline
\end{tabularx}
\end{center}

Note that for $a_f(p)=\alpha(p)+\overline{\alpha(p)}=\alpha+\overline{\alpha}$, we have the following relations:
\begin{align*}
&a_{f}(p^2)=\alpha^2+\overline{\alpha}^2=\alpha^2+1+\overline{\alpha}^2-1=a_{Sym^2 f}(p) - a_{\Lambda^2 f}(p)\\
&a_{f \times f}(p)=a_f(p)^2=\alpha^2+2+\overline{\alpha}^2=\alpha^2+1+\overline{\alpha}^2+1=a_{Sym^2f}(p)+a_{\Lambda^2f}(p).
\end{align*}
Since $f$ is self-dual, $L(s, f \times f)$ has a simple pole at $s=1$ and $L(s, f, \Lambda^2)=\prod_{q|Q_E}\left(1-\frac{1}{p^s} \right)\zeta(s)$ also has a simple pole at $s=1$, from the relations above, $L(s, Sym^2f)$ is entire. Hence, under GRH we have
\begin{align}\label{theta-sum}
\theta_f(x)=\sum_{n \leq x}a_f(p^2)\log p = - x+ O\left( x^{\frac 12}(\log x)(\log(x^3 c_f)) \right)
\end{align} 
for some constant $c_f>0$ \cite[Theorem 5.15]{IK}. 

By partial summation, we have
\begin{lem} \label{first est}
\begin{align*}
\sum_{F \in S_{d,Q_E}(X)} \log |D_F| = |S_{d,Q_E}(X)|\log X + O_{Q_E}(X).
\end{align*}
\end{lem}
By Lemma \ref{first est}, we can estimate the first sum in $(\ref{S_d one-level})$:
\begin{align*}
\frac{2 \widehat{\phi}(0)}{|S_{d,Q_E}(X)|\log L}\sum_{F \in S_{d,Q_E}(X)} \log |D_F|=\widehat{\phi}(0)+ O\left( \frac{1}{\log X} \right).
\end{align*}

To control the sum $S_1$, we need the following lemma. 

\begin{lem} \label{s1-sum}
\begin{align}
\sum_{F \in S_{d,Q_E}(X)} a_\rho(p)=O_{Q_E}\left(X^{\frac 56 } +p^{\frac{2}{3}} X^{ \frac 23 + \epsilon} \right).
\end{align}
\end{lem}
\begin{proof}
By $(\ref{s3-count})$ and the table , 
\begin{align*}
\sum_{F \in S_{3,Q_E}(X)} a_\rho(p)&= c_3X \prod_{q | Q_E}\frac{1}{1+f(q)}    \cdot \left[\frac{1\times 2 + 0 \times 3 +(-1)\times 2}{|S_3|(1+f(p))} \right]   + O_{Q_E}\left(X^{\frac 56 } + p^{\frac{16}{9}} X^{ \frac 79 + \epsilon} \right),\\
&=O_{Q_E}\left(X^{\frac 56 } + p^{\frac{2}{3}} X^{ \frac 23 + \epsilon} \right).
\end{align*}
\end{proof}
Again by $(\ref{s3-count})$ and the table we can show that 
\begin{lem} \label{s2-sum}
\begin{align}
\sum_{F \in S_{3,Q_E}(X)} a_\rho(p^2)=c_{3,Q_E}X + O_{Q_E}\left( \frac 1p X + X^{\frac 56} + p^{\frac{2}{3}} X^{ \frac 23 + \epsilon}   \right). 
\end{align}
\end{lem}

Assume that support of $\widehat{\phi} \subset [-\sigma, \sigma]$ for some $\sigma < \frac{1}{7}$.  
By Lemma \ref{s1-sum} and $(\ref{F-size})$,
\begin{align}\label{S1 est}
S_1 \ll  \frac{X^{\frac 56}}{X \log X}\sum_{ p\leq X^{2\sigma}}\frac{\log p}{p^\frac{1}{2}} + \frac{X^{\frac 23 +\epsilon}}{X \log X}\sum_{ p\leq X^{2\sigma}} p^{\frac{2}{3}-\frac{1}{2}} \log p  \ll \frac{X^{\frac 56 + \sigma} }{X \log X} + \frac{X^{\frac 23 + \frac{7\sigma}{3} + \epsilon}}{X \log X} \ll \frac{1}{\log X}.
\end{align}

For $S_2$, we have
\begin{align*}
S_2&=-\frac{2}{|S_{d,Q_E}(X)|\log L}\sum_{p} \frac{\log p}{p} a_{f}(p^2)\widehat{\phi}\left( \frac{2 \log p}{\log L} \right)\left( c_{3,Q_E}X + O_{Q_E}\left( \frac 1p X + X^{\frac 56} + p^{\frac{2}{3}} X^{ \frac 23 + \epsilon}  \right) \right) \\
&=-\frac{c_{3,Q_E} X}{|S_{d,Q_E}(X)|}\sum_{p} \frac{2 a_{f}(p^2) \log p}{p \log L} \widehat{\phi}\left( \frac{2 \log p}{\log L} \right)+O_{Q_E}\left( \frac{1}{\log X} + \frac{X^{\frac 56} \log\log X }{X} + \frac{X^{\frac 23 + \frac{2\sigma}{3}+\epsilon}}{X \log X}  \right).
\end{align*}

By summation by parts we have
\begin{align*}
\sum_{p}\widehat{\phi}\left( \frac{2 \log p}{\log L} \right) \frac{ 2 a_f(p^2) \log p}{p\log L}=\int_1^\infty \widehat{\phi}\left( \frac{2\log t}{\log L} \right) \frac{2 d\theta_f(t)}{t \log L}.
\end{align*}
Using $(\ref{theta-sum})$, we can show that 
\begin{align} \label{S2 est}
S_2=-\frac{c_{3,Q_E} X}{|S_{d,Q_E}(X)|}\sum_{p}\widehat{\phi}\left( \frac{2 \log p}{\log L} \right) \frac{2a_f(p^2) \log p}{p\log L} +O_{Q_E}\left( \frac{1}{\log X} \right) =\frac{1}{2} \phi(0)+ O_{Q_E}\left( \frac{1}{\log X} \right).
\end{align}
By Lemma \ref{first est}, $(\ref{S1 est})$ and $(\ref{S2 est})$, we establish the one-level density $(\ref{one-level-S-d})$ for any $\sigma < \frac{1}{7}$. 

For $d=4$ and $d=5$, we choose $\sigma=1/864$ and $1/2400$ by a work of the author and Kim \cite{CKcnt} respectively. However, this gives a poor bound. 
 
\section{Acknowledgment}
The author appreciates the anonymous referee for his/her careful reading and helpful suggestions and thanks Myungjun Yu for providing the example of an elliptic curve in the introduction and Keunyoung Jeong for many useful discussion.

\end{document}